\documentclass[11pt]{amsart}
\usepackage{amsfonts}
\usepackage{ifthen}
\usepackage{amsthm}
\usepackage{amsmath}
\usepackage{graphicx}
\usepackage{amscd,amssymb,amsthm,verbatim}

\newcounter{minutes}
\divide\time by 60
\newcounter{hours}
\multiply\time by 60 \addtocounter{minutes}{-\time}

\title[V. N. Dubinin, M. Vuorinen/Ahlfors-Beurling conformal invariant]
{Ahlfors-Beurling conformal invariant and relative capacity of compact sets}

\author[]{Vladimir N. Dubinin$\dagger$}
\author[]{Matti Vuorinen$\ddagger$}

\address{Far Eastern Federal University, Vladivostok, Russia} \email{dubinin@iam.dvo.ru}

\address{Department of Mathematics, University of Turku, Turku 20014,
Finland} \email{vuorinen@utu.fi}

\thanks{$\dagger$The research of this author was supported by the
Russian Foundation for Basic Research, project 11-01-00038}
\thanks{$\ddagger$ Supported by the Academy of Finland, project 2600066611}

\newtheorem{theorem}[equation]{Theorem}

\newtheorem{remark}[equation]{Remark}
\newcounter{propcounter}
\newtheorem{property}[propcounter]{Property}

\numberwithin{equation}{section}

\newcommand{\rc}{{\rm rel cap}{\,}}
\newcommand{\hc}{{\rm  hcap}{\,}}

\begin{document}

\def\thefootnote{}
\footnotetext{ \texttt{\tiny File:~\jobname .tex,
          printed: \number\year-\number\month-\number\day,
          \thehours.\ifnum\theminutes<10{0}\fi\theminutes}
} \makeatletter\def\thefootnote{\@arabic\c@footnote}\makeatother

\maketitle

\begin{abstract}
For a given domain $D$ in the extended complex plane $\overline{\mathbb C}$
with an accessible boundary point  $z_0 \in \partial D$ and for a
subset $E \subset {D},$ relatively closed
w.r.t. $D\,,$ we define the relative capacity $\rc E$ as a
coefficient in the asymptotic expansion of the Ahlfors-Beurling
conformal invariant $r(D\setminus E,z)/r(D, z)$ when $z$ approaches
the point $z_0\,.$ Here $r(G,z)$  denotes the inner radius at $z$ of the
connected component of the set $G$ containing the point $z\,.$
The asymptotic behavior of this quotient is established.
Further, it is shown that in the
case when the domain $D$ is the upper half plane and $z_0=\infty$
the capacity $\rc E$ coincides with the well-known
half-plane capacity ${\hc} E\,.$  Some properties of the
relative capacity are proven, including the behavior of this
capacity under various forms of symmetrization and under some other
geometric transformations. Some applications to bounded holomorphic
functions of the unit disk are given.

\end{abstract}


\section{Introduction }
Let $D_1$ and $D_2$ be domains having Green functions in the extended complex plane $ \overline {\mathbb C} ,$  and let the point $z \in D_1 \subset D_2\,.$ We denote by $r(D_k,z)$ the inner radius of the
domain $D_k, k=1,2,$ at the point $z$ (see e.g. \cite{h}, \cite{d0}).
The quotient $r(D_1,z)/r(D_2,z)$ is conformal invariant in the sense
that for every conformal map $f$ of $D_2$ we have the equality
\begin{equation*} 
\frac{r(D_1,z)}{r(D_2,z)} =
\frac{r(f(D_1),f(z))}{r(f(D_2),f(z))} \, .
\end{equation*}
The study of this kind of invariant expressions goes back to the
works of Ahlfors and Beur\-ling \cite[p.436]{a}. Some significant
applications to geometric theory of functions are given
in (\cite{a,ab,he,o,p}). In this paper we study the behavior of the
invariant $r(D_1,z)/r(D_2,z)$ when the point $z$ tends to a given
common boundary point of the domains $D_1$ and $D_2\,.$ More precisely, we investigate the following situation. Let $D$ be a domain in  $ \overline {\mathbb C} $, and let $z_0$ be an accessible
boundary point of the domain $D\,.$ Suppose that in a neighborhood of the point $z_0$ the boundary $\partial D$ is represented by
an analytic arc $\gamma$ (in the case $z_0= \infty$ it is required that the image of the arc $\gamma$ under the mapping $z \mapsto 1/z$ be analytic.)
Consider an arbitrary set $E \subset D,$ relative closed with respect to $D$ such that the inner distance $\rho(E, z_0)$ from the point $z_0$ to the set $E$ with respect to the domain $D$ is positive.
In the case of a finite point $z_0,$ {\it the relative capacity} $\rc  E$ of the set $E$ is defined via the following
asymptotic expansion
\begin{equation} \label{1.2}
\frac{r(D \setminus E,z)}{r(D,z)} = 1-  2 \,(\rc  E) |z-z_0|^2+ o(|z-z_0|^2)\,, \quad z \to z_0\,,
\end{equation}
where the approach of $z$ to the point $z_0$ takes place along an arbitrary
arc in $D\,$ perpendicular to $\gamma$ at the point $z_0\,.$
Here $r(G,z)$  denotes the inner radius of the
connected component of the set $G$ containing the point $z\,.$
If the point $z_0=\infty,$ then the parameter $|z-z_0|$ in the
definition of the relative capacity $\rc  E$ is replaced
by $|1/z|\,.$ We say that $\rc  E$ is {\it the relative capacity of the set} $E$ with respect to the domain $D$ at the point
$z_0\,.$ We shall establish the asymptotic expansion
\eqref{1.2} in Section 2. It will also be shown that in the case when the domain $D$ is the upper half plane $H$ and the point $z_0=\infty$ we have
\begin{equation*} 
\rc  E = {\hc}\,E \,,
\end{equation*}
where ${\hc}\,E $ is {\it the half-plane capacity} of the set
$E$ \cite{l1} \,. In the case when the set $H \setminus E$ is a simply-connected domain, the equality of the relative and half-plane
capacities was proven in \cite{d1}, with essential use of conformal
mapping (cf. \cite[formula (3)]{d1}). We recall further that the
notion of half-plane capacity arises naturally in statistical physics
when examining the Schramm-L\"owner equations (see G. Lawler \cite{l1,l2,l3}). There
are several definitions of this capacity in the literature. For example, let $G= H \setminus E$ and
\begin{equation*} 
\phi_G(z)= {\rm Im} z - {\bf E}^z[{\rm Im} (B_{\tau_G})]
\end{equation*}
where $B_t$ is a standard Brownian motion and $ \tau_G= \inf \{  t: B_t \notin G \}\,,$
and ${\bf E}^z$ is the mathematical expectation.
Then $\phi_G$ is a positive harmonic function on $G$ that vanishes at the regular points of the boundary $\partial G$ and such that
\begin{equation*} 
\phi_G(z)= {\rm Im} z + O(|z|^{-1})\,, \quad z \to \infty \,.
\end{equation*}
The half-plane capacity (at infinity) of $E$ is defined by

\begin{equation} \label{1.6}
\phi_G(z)= {\rm Im} (z + \frac{{\hc}\, E}{z})+ o(|z|^{-1})\,, \quad z \to \infty \,.
\end{equation}

\noindent
(see \cite[Lecture 2]{l2}). Some new geometric properties of the half-plane capacity were proved in the recent paper \cite{d1}. Following this paper we shall examine the properties of
the relative capacity ${\rc}\,E$ under various geometric transformations of the
set $E$ in the case when the domain $D$ is the disk and the point
$z_0=1\,$(Section 3). Our results do not reduce
 to the corresponding statements of \cite{d1} in the particular case $D=H, z_0=\infty\,,$
because the symmetrization procedures applied here are not invariant under M\"obius transformations. Finally,
in Section 4 there are given some applications of the introduced capacity to holomorphic
functions. The results obtained here turn out to be effective in the study of a boundary
version of the Schwarz lemma involving the Schwarzian derivative (see, e.g. \cite{s,tv,d2}). After the completion of the writing of this
paper, we learned about a very recent paper of S. Rohde and C. Wong, also studying half-plane capacity \cite{rw}.

\section{Existence and properties of relative capacity}

Following the proof of Lemma 1 in the paper \cite{d1} we first establish the
existence of the asymptotic expansion \eqref{1.2}.

\begin{theorem} \label{2.1} Let $D$ be a domain in the extended complex plane $\overline{\mathbb C}$, and let $z_0$ be a finite accessible
boundary point of the domain $D\,.$ Suppose that in some neighborhood of   $z_0$ the
boundary $\partial D$ is represented by an analytic arc $\gamma\,.$ Then for every
set $E$ relatively closed with respect to $D\,$ and with $\rho(E,z_0)>0\,,$ the following expansion holds
\begin{equation}
\frac{r(D \setminus E, z)}{r(D , z)} = 1 - c |z- z_0|^2 +o(|z- z_0|^2 )\,, \quad z \to z_0\,,
\end{equation}
where the convergence of $z$ to the point $z_0$ takes place along a path, perpendicular
to the arc $\gamma$ at $z_0\,$ and where $c\ge 0$ is a constant, depending only on the set
$E$\,, the domain $D\,,$ and on the point $z_0\,.$
\end{theorem}

\begin{proof}
We consider the function $f$ mapping  $D$ conformally and univalently onto a domain
$f(D)\,$ lying in the upper half plane $H\,.$  Extending $f$ to the arc $\gamma$
in the sense of the boundary correspondence we may assume that $f(\gamma)$ is a finite interval of the real
axis and that $f(z_0)=0, |f'(z_0)|=1\,.$  Let $g(w, \zeta)$ be the Green function of the
connected component of the symmetric set
$$
 G= \{     w:  w \in  f(\gamma)  \cup  f(D) \setminus  f(E)  \,\quad {\rm or}  \, \quad
\overline{w}  \in  f(D) \setminus  f(E) \}\,,
$$
that contains the origin and let $h(w,\zeta)=\log |w-\zeta| + g(w,\zeta) \,$ be the
regular part of this function. By the symmetry of $G\,,$  we have
 $$
g_{f(D) \setminus f(E)}(w, f(z))\equiv g(w, f(z))- g(w, \overline{f(z)}) \,.
$$
Adding $\log |w-f(z)|$ to both sides and letting $w\to f(z)\,,$ this
relation gives
$$
\log r(f(D)\setminus f(E), f(z)) = \log r(G, f(z)) -g(f(z), \overline{f(z)})\,,
$$
and hence
\begin{equation*} 
r(f(D)\setminus f(E), f(z))=  r(G, f(z)) \exp \{ -g(f(z), \overline{f(z)}) \}=
\end{equation*}
$$
 \{ r(G, f(z)) \,  r(G, \overline{f(z)})  \exp \{ -2g(f(z), \overline{f(z)}) \}   \}^{1/2} =
$$
\begin{equation} \label{2.4a}
 \exp \{ \frac{1}{2}  \left[  h(f(z), f(z)) + h(\overline{f(z)}, \overline{f(z)}) - 2  h({f(z)}, \overline{f(z)})     +2 \log |f(z)-  \overline{f(z)}| \right] \} \,.
\end{equation}
We consider the function $h(w,\zeta)$ as a function of four real arguments
$h(u,v,\xi, \eta)$ and introduce the notation $f(z) = \Delta u + i \Delta v  \,.$ It is clear that when $z \to z_0$ along an arc
perpendicular to $\gamma\,,$ we have $\Delta v\to 0 $  and
$\Delta u =o(\Delta v)\,,  $  $\Delta v\to 0 \,. $ The symmetric
difference enclosed within the square brackets in \eqref{2.4a}
has the following form in the new notation

\begin{equation} \label{2.4}
 h(\Delta u, \Delta v, \Delta u, \Delta v)+
 h(\Delta  u, -\Delta v, \Delta u, -\Delta v)-
 2 h(\Delta  u, \Delta v, \Delta u, -\Delta v) \,.
\end{equation}

\noindent
In view of the symmetry of the set $G$ with respect to the real
axis we see that
$$
h(u,v,\xi,\eta) = h(u,-v,\xi,-\eta) \,.
$$
Therefore at the point $(0,0,0,0)$
$$
\frac{\partial h}{\partial v} = \frac{\partial h}{\partial \eta} = 0 \,.
$$
Applying this fact and applying Taylor's formula in a neighborhood of the point $(0,0,0,0)\,,$ we have
$$
h(\Delta u, \Delta v, \Delta u, \Delta v) -h(0,0,0,0)=
\frac{\partial h}{\partial u} \Delta u + \frac{\partial h}{\partial \xi} \Delta u +
\frac{1}{2}\left[
\frac{\partial^2 h}{\partial u^2} (\Delta u)^2+ \right.
$$
$$
\frac{\partial^2 h}{\partial v^2} (\Delta v)^2  +
\frac{\partial^2 h}{\partial \xi^2}  (\Delta u)^2
+\frac{\partial^2 h}{\partial \eta^2}  (\Delta v)^2
+2\left( \frac{\partial h^2}{ \partial u \partial v  } \Delta u  \Delta v
  +
\frac{\partial^2 h}{\partial u \partial \xi } (\Delta u)^2 \right.
$$
$$ \left. \left.
+\frac{\partial^2 h}{\partial u \partial \eta} \Delta u \Delta v + 
\frac{\partial^2 h}{\partial u \partial \xi} \Delta u \Delta v+\frac{\partial^2 h}{\partial v \partial \eta} (\Delta v)^2  +
\frac{\partial^2 h}{\partial \xi \partial \eta} \Delta u \Delta v \right)\right]+ o((\Delta v)^2), \quad \Delta v \to 0\,,
$$
$$
h(\Delta u, -\Delta v, \Delta u, -\Delta v) -h(0,0,0,0)=
\frac{\partial h}{\partial u} \Delta u + \frac{\partial h}{\partial \xi} \Delta u +
\frac{1}{2}\left[
\frac{\partial^2 h}{\partial u^2} (\Delta u)^2+ \right.
$$
$$
\frac{\partial^2 h}{\partial v^2} (\Delta v)^2  +
\frac{\partial^2 h}{\partial \xi^2}  (\Delta u)^2
+\frac{\partial^2 h}{\partial \eta^2}  (\Delta v)^2
+2\left( -\frac{\partial^2 h}{ \partial u \partial v  } \Delta u  \Delta v
  +
\frac{\partial^2 h}{\partial u \partial \xi } (\Delta u)^2 \right.
$$
$$ \left. \left. -\frac{\partial^2 h}{\partial u \partial \eta} \Delta u \Delta v -\frac{\partial^2 h}{\partial u \partial \xi} \Delta u \Delta v+\frac{\partial^2 h}{\partial v \partial \eta} (\Delta v)^2  -
\frac{\partial^2 h}{\partial \xi \partial \eta} \Delta u \Delta v \right) \right] +o((\Delta v)^2), \quad \Delta v \to 0\,,
$$
$$
h(\Delta u, \Delta v, \Delta u, -\Delta v) -h(0,0,0,0)=
\frac{\partial h}{\partial u} \Delta u + \frac{\partial h}{\partial \xi} \Delta u +
\frac{1}{2} \left[\frac{\partial^2 h}{\partial u^2} (\Delta u)^2+ \right.
$$
$$
\frac{\partial^2 h}{\partial v^2} (\Delta v)^2  +
\frac{\partial^2 h}{\partial \xi^2}  (\Delta u)^2
+\frac{\partial^2 h}{\partial \eta^2}  (\Delta v)^2
+2\left( \frac{\partial h^2}{ \partial u \partial v  } \Delta u  \Delta v
  +
\frac{\partial^2 h}{\partial u \partial \xi } (\Delta u)^2 -\frac{\partial^2 h}{\partial u \partial \eta} \Delta u \Delta v +
\right.
$$
$$
\left. \left.
\frac{\partial^2 h}{\partial u \partial \xi} \Delta u \Delta v-\frac{\partial^2 h}{\partial v \partial \eta} (\Delta v)^2  -
\frac{\partial^2 h}{\partial \xi \partial \eta} \Delta u \Delta v \right) \right] +o((\Delta v)^2),  \quad \Delta v \to 0\,.
$$
Taking into account the expression in \eqref{2.4} and substituting
into  \eqref{2.4a} we arrive at the identity
$$
r(f(D)\setminus f(E), f(z)) = |2  \Delta v|
\exp \{ 2\,  \frac{\partial^2 h}{\partial v \partial \eta} (\Delta v)^2   +o((\Delta v)^2)\} =$$
$$
2 \Delta v  + c_1 (\Delta v)^3 + o((\Delta v)^3)\,,  \quad \Delta v \to 0\,.
$$
Here the constant
$$ c_1 = 4 \frac{\partial^2 h}{\partial v \partial \eta}$$
does not depend on the choice of the arc $\gamma$ because the
left hand side is independent of this arc.

Repeating the preceding argument and replacing  $f(D)\setminus f(E)$ with $f(D)$ we arrive at a similar relation
$$
r( f(D), f(z))= 2 \Delta v + c_2 \,(\Delta v)^2 + o((\Delta v)^3), \quad \Delta v \to 0\,.
$$
Therefore
$$
\frac{r(D \setminus E,z)}{r(D ,z)} =
{ \frac{r(f(D) \setminus f(E),f(z))}{r(f(D) ,f(z))}}
= 1+\frac{1}{2} (c_1-c_2) (\Delta v)^2 +  o((\Delta v)^2)=
$$
$$
1-c |z-z_0|^2 +  o(|z-z_0|^2)\,, \quad z \to z_0 \,.
$$
\bigskip

\noindent
Because the expression on the left hand side does not depend
on the choice of the function $f$ and does not exceed $1\,,$
we see that the constant $c :=-\frac{1}{2}(c_1-c_2)\ge 0$ is
independent of $f\,.$ The theorem is proved.
\end{proof}
\medskip

The asymptotic expansion in the case $z_0=\infty$ is contained
in what was proved above with the change of variable $z -z_0 \mapsto 1/z\,.$ Note that we did not only prove the existence of the
expansion \eqref{1.2}, but also established a representation of the relative capacity in terms of the Green functions of the domains in question.

\begin{remark} From the above proof it is clear that the requirement of the analyticity of the arc $\gamma$ can be weakened. It is
enough to require that the arc $\gamma$ have a tangent at the point $z_0$  and the function $f$ has an extension to the arc $\gamma$ and the following expansion holds
$$  f(z)-f(z_0) = a(z-z_0) +o((z-z_0)),\, \quad z \in D \cup \gamma, \quad z \to z_0\,,$$
where $a$ is some constant.
\end{remark}

Next we prove that in the case $D=H$ and $z_0=\infty$ the
relative capacity and the half-plane capacity coincide.

\begin{theorem}
The capacity $\rc \, E$ of a bounded set,
relatively closed in the
half-plane $H\,,$
 at the point $z_0=\infty$ is equal
to the half plane capacity $\hc \, E\,.$
\end{theorem}

\begin{proof}  
The expansions \eqref{1.2} and \eqref{1.6} take the following
form after the change of variable $\zeta= -1/z$
\begin{equation} \label{2.5}
\frac{r(\tilde{G}, i \eta)}{2 \eta} = 1 -2 \, (\rc \, E) \, \eta^2 + o(\eta^2)\,, \quad  \eta \to 0\,,
\end{equation}

\begin{equation} \label{2.6}
\tilde{\phi}(\zeta)= {\phi}_G(-1/\zeta)=
- {\rm Im} \,1/\zeta - {\rm Im} \,(\hc \, E )\zeta + o(\zeta)\,,
 \quad  \zeta \to 0\,,
\end{equation}
where $\tilde{G}= \{ \zeta\,:\, -1/\zeta \in G \},\, G= H \setminus E, \zeta= \xi+i \eta \,.$ Consider the harmonic function
$$
u(\zeta)= g_{\tilde{G} }(\zeta, i \eta) - g_{H }(\zeta, i \eta)
-2 \, \eta \left[  \tilde{\phi}(\zeta)+  {\rm Im} \frac{1}{\zeta} \right] \,,  \quad \zeta \in \tilde{G}\,.
$$
At the regular boundary points of the boundary of the domain
$\tilde{G}$ we have
$$
u(\zeta)= \log \left|  \frac{\zeta -i \eta}{\zeta +i \eta}\right|
-2 \eta \, {\rm Im} \frac{1}{\zeta} =
 \log \left|  \frac{1 +i z \eta}{1- i z \eta}\right| + 2 \,\eta \, {\rm Im} \,z=
$$
\medskip
$$
\log \left| (1+ iz \eta)(1+ iz \eta - z^2 \eta^2 + o(\eta^2))   \right| + 2 \,\eta \, {\rm Im} \,z=
$$
\medskip
$$
\log \left| 1+ 2 \,i\,z \eta -2\, z^2 \eta^2 + o(\eta^2)   \right| + 2 \,\eta \, {\rm Im} \,z=
$$
\medskip
$$
\log \left| 1- 2 \,y \eta +2\,( y^2 -x^2) \eta^2 +i[2x \eta -4xy \eta^2] + o(\eta^2))   \right| + 2 \,\eta \, y=
$$
$$
\frac{1}{2} \, \log (1 - 4 y \eta + 8 y^2 \eta^2 +  o(\eta^2))
+ 2 \eta y = o(\eta^2), \quad  \eta \to 0\,,
$$
$(z = x+ iy)\,.$ According to the maximum principle we have
$$ u(i\eta)= o(\eta^2), \quad  \eta \to 0\,.  $$
This relation together with \eqref{2.6} give
$$
 \log \frac{r(\tilde{G},i \eta)}{ r(H,i \eta)} + 2 \,(\hc \, E)
 \eta^2 =  o(\eta^2)\,,
$$
which in view of \eqref{2.5} gives
$$  \rc \, E = \hc \, E \,.$$
The theorem is proved.
\end{proof}

We next record some immediate properties of the relative capacity that
follow from the properties of the inner radius and the expansion
\eqref{1.2} (cf. \cite{l1}). In what follows the domain $D$ and the point $z_0$ are fixed.



\begin{property} \label{2.8} {\rm (Monotonicity)} If $E_1 \subset E_2\,,$ then
$$  \rc \, E_1 \le \rc \, E_2 \,.$$
\end{property}
\begin{proof}
The proof is a consequence of the monotonicity of the inner radius
$$
r(D \setminus E_1, z) \ge r(D\setminus E_2, z)
$$
which follows from the inclusion $ D \setminus E_2 \subset D \setminus E_1\,.$
\end{proof}

\begin{property} \label{2.9}  {\rm  (Choquet's inequality)} For all
 $E_1 , E_2\,,$ the inequality
$$  \rc \, E_1 + \rc \, E_2  \ge \rc \, (E_1 \cup E_2) +
\rc \, (E_1 \cap E_2)
\,.$$
holds.
\end{property}
\begin{proof}
The proof follows from an inequality of Renggli (see \cite{r}) and
the formula \eqref{1.2}.
\end{proof}

Let the domain $D$ be symmetric with respect to the imaginary axis and let
$z_0$ be a point on this axis, $z_0 \in \partial D\,.$ For a given set $E \subset D$ we define the set
$$
 P \, E =(E \cup E^*)^+ \cup ( E \cap E^*)^- \,,
$$
where $A^*$ denotes a set symmetric to $A$ with respect to the imaginary
axis, whereas  $A^+(A^-)$ is the intersection of $A$ with the right (left)
closed half plane.

\begin{property} \label{2.10} {\rm  (Polarization principle)} The following
inequality holds
$$  \rc \, E  \ge \rc \, P \, E
\,.$$
\end{property}
\begin{proof}
Consider the set
$$
P_c(D \setminus E) = ((D \setminus E) \cup (D \setminus E)^*)^- \cup
((D \setminus E) \cap (D \setminus E)^*)^+ \,.
$$
It is easy to see that
$$
P_c(D \setminus E)= D \setminus P\, E.
$$
According to Corollary 1.2 of \cite{d0}
$$
r(P_c(D \setminus E), z) \ge r(D \setminus E, z)
$$
for every point of the imaginary axis, contained in $D \setminus E\,.$
It remains just to apply the formula \eqref{1.2}.
\end{proof}

\begin{property} \label{2.11} {\rm (Composition principle)} Under the hypotheses
and notations introduced before Property {\rm \ref{2.10}} we have the following inequality
$$ 2\, \rc \, E  \ge \rc \, (E^+ \cup (E^+)^*) +
\rc \, (E^- \cup (E^-)^*)
\,.$$
\end{property}
\begin{proof}
According to Theorem 1.9 of \cite{d1}
$$
r^2(D\setminus E, z) \le r(D \setminus  (E^+  \cup (E^+)^* ),z)
r(D \setminus  (E^-  \cup (E^-)^* ),z)
$$
for the points of the imaginary axis. For both parts of the inequality for
$r^2(D,z)$ we apply the formula \eqref{1.2}, { and arrive at the desired conclusion.}
\end{proof}

\begin{property} \label{2.11}
Given a relatively closed subset $E$ of $D\,,$ there exists a sequence
of open sets $  \{ B_n\}_{n=1}^{\infty}$ such that $E \subset B_n,
n=1,2,\dots,$  and
$$ \lim_{n\to \infty} \rc(\overline{B}_n  \cap D) = \rc \, E
\,.$$
\end{property}
\begin{proof}
The proof is essentially contained in the proof of a particular
case (Lemma 3 of the paper \cite{d1}) noting the representation
of the relative capacity in terms of the Green function
(see the end of the proof of Theorem \ref{2.1} ).
\end{proof}

\begin{property} \label{2.12}
Suppose that $f$ maps a domain $D$ onto a domain $f(D)$ conformally
and univalently such that a boundary point $z_0$ is mapped in the
sense of boundary correspondence to a point $w_0 \in \partial f(D)\,.$ We assume further that in a neighborhood of the
 point $z_0$ the boundary $\partial D$ is an analytic
 arc and the boundary $\partial f(D)$ in a neighborhood of $w_0$ is
 also an analytic arc. If, furthermore, $z_0=w_0=\infty$ and
 $\lim_{z \to \infty} \,f(z)/z =a\,,$ then
 $$
 \rc f(E) = |a|^2 \rc E \,,
 $$
and if $z_0$ and $w_0$ are finite points, then
$$
\rc f(E) = (\rc E)/|f'(z_0)|^2\,,
$$
 for every relatively closed subset $E\subset D \,, \rho(E,z_0)>0\,.$
\end{property}

The proof follows from the formula \eqref{1.2}\,.

\bigskip
\section{The behavior of the relative capacity under geometric
transformations of subsets in the disk}


It would be an interesting problem to study the behavior of the relative capacity under simultaneous geometric
transformations of the domain $D$ and the set $E \subset D\,.$
However, this problem appears to be very difficult. Therefore we restrict ourselves here only to the case when the domain $D$ is
fixed. Furthermore, throughout this section, the domain $D$ is the
unit disk $U= U_z=\{z:  |z|<1\}\,,$ the point $z_0=1,$ $E$ is a
relatively closed subset of $U,$ with $\rho(E, 1) > 0\,,$
and $\rc  E$ stands for the relative capacity of
$E$ with respect to the disk $U$ at the point $z=1\,.$

We recall the definition of the circular symmetrization of closed and open sets with respect to a given ray (see \cite{ps,h,d0}). Given a real number $a$ let $\gamma_r(a)$ be the circle $|z-a|=r$ (for $0<r<\infty$), which degenerates to the point $a$ (or $\infty$) if
$r=0$ ($r=\infty$, resp.). By the {\it circular symmetrization} of a
closed set $F \subset \overline{\mathbb C}$ with respect to the ray
$[-\infty,a]$ we mean the transformation of this set onto a symmetric
set ${\rm Cr}^{-}_{a} F$  which is defined as follows. If, for a given
$0\le r \le \infty\,,$ the 'circle' $\gamma_r(a)$  does not meet the
set $F\,,$ then it has empty intersection with the set ${\rm Cr}^{-}_{a} F$ as well. If  $\gamma_r(a) \subset  F\,,$ then $\gamma_r(a) \subset  {\rm Cr}^{-}_{a} F\,.$ In the remaining cases, the set ${\rm Cr}^{-}_{a} F$ intersects  $\gamma_r(a)$  along a closed arc with center at the ray $[-\infty, a]\,,$ whose linear measure agrees with the measure of the
intersection of $F$ with $\gamma_r(a)\,.$ It is readily verified that the set ${\rm Cr}^{-}_{a} F$  is closed. In the same way we define the
the result of the circular symmetrization ${\rm Cr}^{+}_{a} F$ with respect to the
ray $[a,\infty].$ The only difference consists of the fact that the center of the closed arc is now contained on the ray  $[a,\infty].$
In the same way, we define the circular symmetrizations  ${\rm Cr}^{+}_{a} \,,\,{\rm Cr}^{-}_{a}$ of open sets, now taking an open arc in place of a closed arc.



\begin{theorem} \label{3.1} The following inequalities hold
\begin{equation} \label{3.3}
  \rc E \ge \rc  [  ({\rm Cr}^{-}_{o} \overline{E})\cap U]
\end{equation}
\begin{equation} \label{3.4}
  \rc E \ge \rc  [  ({\rm Cr}^{-}_{a} (E \cup \Sigma))\setminus \Sigma]\,,
  \quad (a \le 0)\,,
\end{equation}
\begin{equation} \label{3.5}
  \rc E \ge \rc  [  ({\rm Cr}^{+}_{a} (E \cup \Sigma))\setminus \Sigma]\,,
  \quad (a \ge 1)\,,
\end{equation}
where $\Sigma= \{ z: |z|\ge 1  \}\,.$
\end{theorem}

\begin{proof} All three inequalities can be proved in a unified
way, applying formula \eqref{1.2} and Polya's result on the behavior of the
inner radius of a domain under circular symmetrization (see \cite{h,d0}).
For instance, from Polya's theorem it follows that for every $x,\,0<x<1\,,$
$$
r(U\setminus E, x) \le r({\rm Cr}^+_{o}(U \setminus E), x)
=  r( U \setminus [{\rm Cr}^-_{o}( \overline{E}) \cap U], x)\,.
$$
It remains to apply the formula \eqref{1.2} with $D=U\,.$ Analogously,
for $a\le 0$
$$
r(U\setminus E, x) \le r({\rm Cr}^+_{a}(U \setminus E), x)
=  r( U \setminus [{\rm Cr}^-_{a}( {E} \cup \Sigma) \setminus \Sigma], x)\,.
$$
and for $a \ge 1$ we have
$$
r(U\setminus E, x) \le r({\rm Cr}^-_{a}(U \setminus E), x)
=  r( U \setminus [{\rm Cr}^+_{a}( {E} \cup \Sigma) \setminus \Sigma], x)\,.
$$
The theorem is proved.
\end{proof}

Figure 1 provides an illustration of a set $E$ and Figures 2,3,4 illustrate
its deformation under each of the transformations in the inequalities
\eqref{3.3}, \eqref{3.4}, \eqref{3.5}, respectively.
\begin{center}
\begin{figure}[h]
\includegraphics{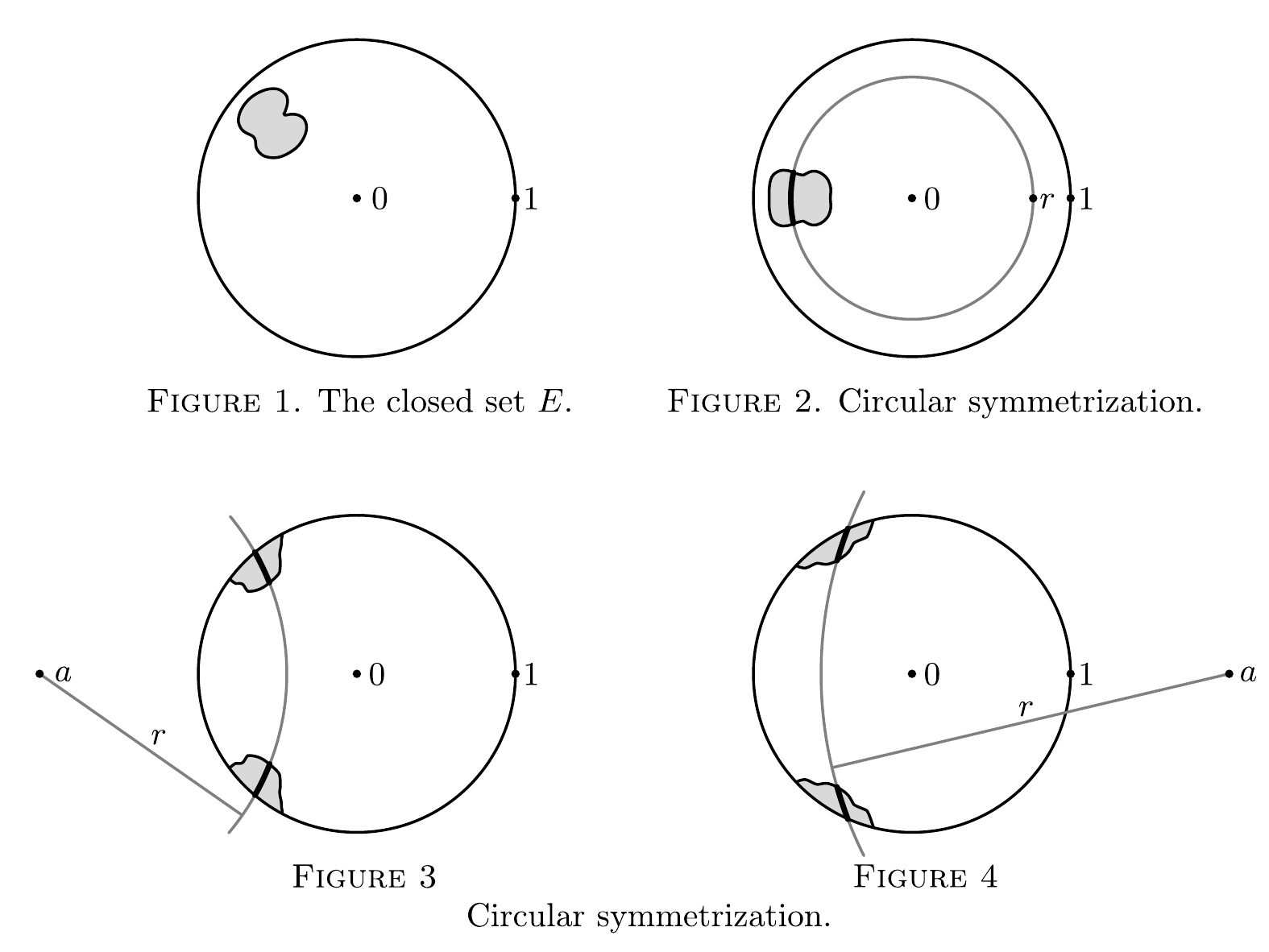}
\end{figure}
\end{center}
The {\it Steiner symmetrization } of an open set $B$ with respect to
the real axis is defined by
$$
{\rm St}\, B = \{ x+iy: B \cap \lambda(x) \neq \emptyset, \,
2 |y|< \mu(B \cap \lambda(x) ) \}\,,
$$
where $\lambda(x)$ is the line ${\rm Re}\, z= x,$ and $\mu(\cdot)$
is the linear Lebesgue measure.

\begin{theorem} \label{3.8}
The following inequality holds
\begin{equation} \label{3.9}
\rc E \ge \rc [  U \setminus {\rm St}\,(U \setminus E)] \,.
\end{equation}
\end{theorem}

\begin{proof}
In the same way as in the proof of Theorem \ref{3.1}, it is enough
to apply a theorem of Polya and Szeg\"o \cite{ps}
\begin{equation}
r(U \setminus E, x) \le r({\rm St}\,(U \setminus E),x)
\end{equation}
and to note the equality
$$   U \setminus [U \setminus {\rm St}\,(U \setminus E)] =
{\rm St}\,(U \setminus E)  \,.$$
The proof is complete.
\end{proof}

\begin{center}
\begin{figure}[h]
\includegraphics{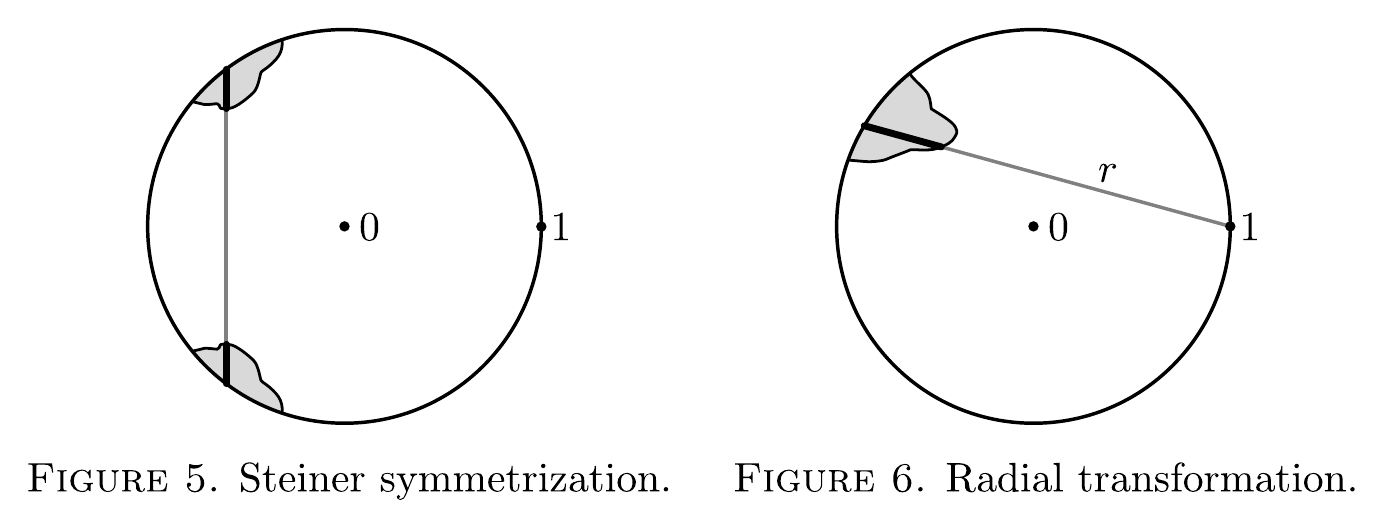}
\end{figure}
\end{center}

We consider the following radial Marcus transformation (see \cite{m1}),
which transforms an open set $B, z_0 \in B$ onto the starlike set
$$
{\rm M}_{z_o}B = \{   z_0  + r e^{i \theta} : 0 \le r <  M(\theta,B)\,,
0 \le \theta \le 2 \pi \}\,,
$$
where
$$
M(\theta,B)=
\rho \, \exp \left(  \int_{F(\rho, \theta, B)} \frac{dr}{r} \right)\, \,,\quad
F(\rho, \theta, B)= \{r : z_0  + r e^{i \theta} \in B\,,
 \rho \le r < \infty  \} \,.
$$
Here $M(\theta,B)$ is independent of $\rho\,$ and
$\{   z:  |z-z_0|\le \rho\} \subset B\,.$
For open sets $B\,,$ lying in the disk $U$ and containing a relative
neighborhood $\{  z:  |z-1|  < \varepsilon\} \cap U$ for some $\varepsilon >0\,$
we define ${\rm M}  B \equiv {\rm M}_1B\,$ as above, but with a constraint on
$\theta: \pi/2 < \theta< 3 \pi/2 \,.$

\begin{theorem} \label{3.10}
The following inequality holds
\begin{equation} \label{3.11}
\rc E \ge \rc [U \setminus {\rm M}(U \setminus E)]\,.
\end{equation}
\end{theorem}

\begin{proof}
Let  $  \{ B_n\}_{n=1}^{\infty}$ be a sequence from Property
\ref{2.11} for the given set $E, D=U\,,$ and let
$E_n= \overline{B}_n \cap U, n=1,2,\dots .$
According to a theorem of Marcus \cite{m1} for a fixed integer $n$ and for values of $x$ sufficiently close to $1, 0<x<1\,,$ we have
$$
r(U\setminus E_n,x) \le r({\rm M}_x(U\setminus E_n),x) =  r(U\setminus
[U\setminus {\rm M}_x(U\setminus E_n)],x)\,.
$$
On the other hand, because of the openness of $B_n,$ one can take $x$
so close  to $1$ that
$$  {\rm M}_x(U\setminus E_n)  \subset {\rm M}(U\setminus E) \,. $$
Therefore
$$
r(U\setminus [U\setminus {\rm M}_x(U\setminus E_n)],x)
\le r(U\setminus
[U\setminus {\rm M}(U\setminus E)],x) \,.
$$
It remains to apply formula \eqref{1.2} and Property \ref{2.11}. The proof is complete.
\end{proof}

\medskip
Let $A= \{ \alpha_k \}_{k=1}^{n}$ be a family of positive numbers
with $\sum_{k=1}^n  \, \alpha_k =1\,. $ We consider an averaging
transformation (see \cite{m2})\,, which assigns the starlike set

$$
{\mathbb R}^{{z_0}}_A \{ B_k \}_{k=1}^n =
\left\{  z_0 + r \,e^{i \theta} : 0 \le r < \Pi_{k=1}^n (M(\theta, B_k))^{\alpha_k} \,, 0 \le \theta \le 2 \pi \right\}
$$

\noindent
to a family of open subsets $B_k, k=1,...,n\,,$ containing a point
$z_0;$ here $M(\theta, B)$ is given above.
By Theorem 2.2 of \cite{m2}
$$
\Pi_{k=1}^n (r( B_k, z_0))^{\alpha_k} \le  r({\mathbb R}^{{z_0}}_A \{ B_k \}_{k=1}^n , z_0) \,.
$$

For a family of open subsets of the disk containing the neighborhood
$\{   z: |z-1| < \varepsilon\} \cap D$ for some $\varepsilon >0$, the set
$$
{\mathbb R}^{}_A \{ B_k \}_{k=1}^n \equiv {{\mathbb R}^{1}}_A \{ B_k \}_{k=1}^n
$$
is defined but under a constraint on $\theta: \, \pi/2<\theta < 3\pi/2\,.$
Repeating the proof of preceding theorem replacing ${\rm M}_x$ with
${\mathbb R}^x_A$ and applying Marcus' theorem and Property
\ref{2.11}, we arrive at the following result.

\begin{theorem} \label{3.12}
For every collection of sets $ \{ E_k \}_{k=1}^n $
and for every family  $A= \{ \alpha_k \}_{k=1}^{n}$ of positive numbers
with  $\sum_{k=1}^n  \, \alpha_k =1\,,  $ the following inequality hold
$$   \sum_{k=1}^n  \, \alpha_k \, \rc \, E_k \ge \rc \, [U \setminus
{\mathbb R}_A \{ U \setminus E_k \}_{k=1}^n  ] \,.
 $$
\end{theorem}

Applying Theorem \ref{3.12} to the case when the collection of sets consists of two sets $\{  E,  \{ z : \overline{z}  \in E\} \}$ and the family of numbers $A= \{  \frac{1}{2}, \frac{1}{2}\}\,,$ we obtain the following theorem.

\begin{theorem} \label{3.13}
The following inequality holds
$$
\rc E \ge \rc {\rm R} E \, ,
$$
$$
{\rm R} E = \left\{  1+ r e^{i \theta} : (  { M} (\theta, U \setminus E)
 {M} (2 \pi -\theta,  U \setminus E))^{1/2} \le r < -2 \cos \theta,
\pi/2 < \theta< 3 \pi/2 \right\} \,.
$$
\end{theorem}

\section{ Applications to bounded holomorphic functions in the disk }
The important role of the Schwarzian derivative in geometric function
theory is well-known. In particular, it is interesting to find geometric estimates of the Schwarzian derivative. In this section are
given estimates of the Schwarzian derivative in terms of the geometry of the image domain of the unit disk under
a bounded holomorphic function. These estimates follow in a unified
way from the connection of the Schwarzian derivative and the relative
capacity at a boundary point. Note that distortion estimates involving the Schwarzian derivative originate, in particular, from generalized rigidity theorems.
For instance, Theorem 2.6 of the paper \cite{tv} can now be understood as
distortion result under the hypothesis that the image of the disk lies within a horocycle. In this context, there occurs also a natural connection with the initial coefficients of the expansion of the function in the neighborhood of a boundary point (see also \cite[  Theorem 4]{s}, \cite[ Proposition 2]{d2}). Here a more general situation is considered, when the image of the disk has a property, characteristic for the application of the symmetrization method \cite{h}. In accordance with what was said above,
we denote by $\mathcal{ B}$ the class of functions  $f$ holomorphic in the disk
$U_z= \{ z : |z|<1 \}\, , f(U_z)\subset U_w$ with the asymptotic expansion
$$
f(z)= 1+ a_1(z-1)+  a_2(z-1)^2 +  a_3(z-1)^3+  \angle o((z-1)^3)
$$
where $a_1>0,$ $ {\rm Re} (2 a_2+ a_1(1-a_1))=0$ and $ \angle o((z-1)^3)$ is an
infinitesimal quantity in comparison with $(z-1)^3$ as $z \to 1$ in any
Stolz angle in $U_z$ with vertex at $z=1\,.$

The condition for the real part of the combination of initial coefficients of the function in the definition of the class $\mathcal{ B}$ is a necessary normalization for obtaining estimates for the Schwarzian at a boundary point of the unit disk. The corresponding condition
for a coefficient of  the holomorphic functions of the upper half plane has the form ${\rm Im\,}\, c_2=0.$  The geometric meaning of this condition is given in \cite{d2}.

Everywhere in what follows
in this section $\rc \, E$ stands for the relative capacity of $E$ with respect to the disk $U$ at the point $z=1\,.$

\begin{theorem} \label{4.1}
For every function $f$ of the class $\mathcal{ B}$ and for every set $E\,,$ closed
with respect to the disk $U_w$ and lying in the complement
$U_w \setminus f(U_z)$
at a positive distance from the point $z=1\,,$ the following inequality  for the Schwarzian derivative holds
$$
- \frac{1}{6} {\rm Re}\, S_f(1):= - {\rm Re}\,  \left(  \frac{a_3}{a_1}-
 \frac{a_2^2}{a_1^2}   \right) \ge a_1^2 \rc \, E \,.
$$
The equality holds for functions $f$ of class $\mathcal{ B}\,,$ mapping conformally and univalently
 the disk $U_z$ onto the domain $U_w \setminus  E \,.$
\end{theorem}

\begin{proof}
In view of the monotonicity of the inner radius and the formula \eqref{1.2} we have for points $x$ on the segment $(0,1)$
$$
\frac{r(f(U_z), f(x))}{  r(U_w, f(x))} \le \frac{r(U_w\setminus E, f(x))}
{  r(U_w, f(x))} = 1 - 2 (\rc  \, E) \, |f(x)-1|^2+ o(|f(x)-1|^2)=
$$
\begin{equation} \label{4.14}
1 - 2\, a_1^2\, (\rc  \, E)(1-x)^2 + o((x-1)^2)\,,\quad x \to 1 \,.
\end{equation}
In the case $f(U_z)=U_w \setminus E$, the equality holds  in \eqref{4.14}\,.
On the other hand, from the result of Hayman \cite[Theorem 4.7]{h} it follows that
\begin{equation} \label{4.15}
\frac{r(f(U_z), f(x))}{  r(U_w, f(x))} \ge
 \frac{|f'(x)|\,r(U_z, x)}{  r(U_w, f(x))} = \frac{|f'(x)|\,(1-x^2)}{1-|f(x)|^2}\, ,
\end{equation}
where, furthermore, the sign of equality holds for univalent functions.
We next find the asymptotic behavior of the right hand side of \eqref{4.15}
$$
 \frac{|f'(x)|\,(1-x^2)}{1-|f(x)|^2}= $$
\medskip
$$  \frac{|a_1-2 a_2(1-x)+ 3 a_3(1-x)^2
+o((1-x)^3)|  (1-x)(2-(1-x))}{ 1- (1-
 a_1(1-x)+ {\rm Re}\, a_2(1-x)^2 -{\rm Re}\, a_3(1-x)^3
+o((1-x)^3))^2 +  o((1-x)^3)}=
$$
\medskip
$$
\frac{|2 a_1+(1-x)(-4a_2-a_1) +(1-x)^2(6 a_3+2 a_2) +o((1-x)^2)|  }{
2a_1+(1-x)(-a_2^2-2 {\rm Re} \, a_2) +(1-x)^2(2 {\rm Re} \, a_3 + 2 a_1 \, {\rm Re} \, a_2)   +o((1-x)^2) }=
$$
\medskip
$$
\left| 1 +(1-x)  \left(-U+V \right)+(1-x)^2 \left[ \frac{3 a_3}{a_1}+ \frac{ a_2}{a_1} -\frac{{\rm Re} a_3}{a_1}-{\rm Re} a_2 - U \, V+ V^2 \right] +o((1-x)^2) \right|\equiv W\,,
$$

\medskip
\noindent
where $U=\frac{2 a_2}{a_1} + \frac{1}{2}\,,  V= \frac{a_1}{2}+
\frac{{\rm Re} \, a_2 }{a_1} \,.$  Next, writing
$$
Y\equiv\frac{3a_3}{a_1}+ \frac{a_2}{a_1}- \frac{{\rm Re} a_3}{a_1}- a_2- \frac{2 a_2 {\rm Re} a_2}{a_1^2}-\frac{a_1}{4}-\frac{{\rm Re} a_2}{2a_1}+\frac{a_1^2}{4}
+\frac{({\rm Re} a_2)^2 }{a_1^2}
$$
we have
$$
W= \left| 1+(1-x)\left[ \frac{-{\rm Re} a_2 +a_1^2-a_1}{2a_1} - i \frac{2 {\rm Im} a_2}{a_1} \right] +Y (1-x)^2 +o((1-x)^2)  \right|=
$$
$$
\left\{ [1+(1-x)^2 {\rm Re}Y +o((1-x)^2) ]^2 +(1-x)^2 \frac{4 ({\rm  Im} a_2)^2}{a_1^2} +o((1-x)^2)  \right\}^{1/2} =
$$
$$
1+(1-x)^2 \left(
 \frac{2{\rm Re} a_3}{a_1}-  \frac{2({\rm Re} a_2)^2}{ a_1^2}+ \frac{2 ({\rm Im} a_2)^2}{a_1^2}
 \right) +o((1-x)^2)
$$
$$
= 1-2 {\rm Re} \left(  - \frac{a_3}{a_1}+ \frac{a_2^2}{a_1^2} \right)(1-x)^2
+o((1-x)^2)\,, \quad x \to 1 \,.
$$
\medskip
Substituting this resulting asymptotic expansion in \eqref{4.15} and
adding \eqref{4.15} and \eqref{4.14} we arrive at the required inequality. The theorem is proved.
\end{proof}

\begin{remark} \label{4.16}
It can be proved that the condition
$$ {\rm Re} \left(2a_2+a_1(1-a_1) \right)=0$$
is essential for Theorem {\rm \ref{4.1}} {{\rm (}cf. \cite[p. 651]{d2}{\rm )}.}
\end{remark}

\begin{remark} \label{4.17} It is well-known that functions of class $\mathcal{ B}$ satisfy
${\rm Im} S_f(1)=0$ {{\rm (}see \cite{s,tv}{\rm )}.} Therefore in the inequality of
Theorem {\rm \ref{4.1}} and in the following applications, the real part of
the Schwarzian could be replaced with $S_f(1)\,.$
\end{remark}

Some applications of Theorem \ref{4.1} to the proof of
geometric properties of holomorphic functions are given, for instance, in the following statements (cf. \cite{s,tv,lln, d2}).

\begin{theorem} \label{4.19}
Let the function $f$ be of class $\mathcal{ B}\,,$ and let the angular Lebesgue
 measure of the intersection of the image of the unit disk $f(U_z)$
with every circle $|w|=r, 0<r<1\,,$
be smaller than or equal to $\alpha,
0 <\alpha< 2\pi\,.$
Then
$$
{\rm Re} \frac{ S_f(1)}{(f'(1))^2} \le - \frac{3 \pi^2}{2 \alpha^2}
\left[(\frac{\alpha}{\pi} -1)^2 +1\right]\,.
$$
The case of equality holds for the function
$$
f_{\alpha}(z)= \left( \frac{z-1+ \sqrt{2z^2 +2}}{z+1} \right)^{\alpha/\pi}
$$
which maps the unit disk $U_z$ conformally and univalently onto the sector
$B(\alpha)= \{  z: |z|<1\,,  |{\rm arg} z| <\alpha/2 \}\,.$
\end{theorem}

\begin{proof}
By Theorems \ref{4.1} and \ref{3.1} (inequality \eqref{3.3}) we have

\begin{equation} \label{4.20}
-\frac{1}{6}{\rm Re} \frac{S_f(1)}{(f'(1))^2} \ge \rc \,E \ge \rc [  ({\rm Cr}^-_{o} \overline{E}) \cap U_w]
\end{equation}
for every set $E\,$ relatively closed  with respect to $U_w\,,$ lying in $U_w \setminus f(U_z)$ and having a positive distance from the point $z=1\,.$ Set
$E=U_w \setminus f(U_z)\,.$ If the distance from $ U_w \setminus f(U_z)$ to the point $z=1$ is zero, then we set
$$E=(U_w \setminus f(U_z))\setminus \{ z: |z-1|< \varepsilon  \}\,,$$
for a sufficiently small $\varepsilon >0\,.$ If $\beta$ is a number with
$0< \alpha < \beta < 2\pi\,,$ then by the hypothesis of the theorem,
we can choose $\varepsilon$ such that the set
$ {\rm Cr}^-_{o} \overline{E}$
{ contains $U_w \setminus B(\beta)\,.$ }
From the monotonicity of the capacity and again by Theorem \ref{4.1} we obtain
$$
\rc\,[ ({\rm Cr}^-_{o} \overline{E}) \cap U_w ] \ge \rc [U_w \setminus B(\beta)]= -\frac{1}{6}{\rm Re} \frac{S_{f_{\beta}}(1)}{(f'_{\beta}(1))^2} \,.
$$
The proof now follows by comparing the two obtained inequalities, computing
the derivative of the function $f_{\beta}$ and letting $\beta\to \alpha \,.$
\end{proof}

\begin{theorem} \label{4.25}
If a univalent function $f$ of the class $\mathcal{ B}$ does not take in the disk $U_z$
a value $w_0 \in U_w\,,$ then it satisfies the inequality \eqref{4.26},
\begin{equation} \label{4.26}
{\rm Re} \frac{ S_f(1)}{(f'(1))^2} \le -\frac{3}{4}
\left(\frac{1-\rho}{1+\rho} \right)^2  \,,
\end{equation}
where $\rho= |w_o| \,.$ Here equality holds, for instance, for the function of Pick, $w=f(z;\rho)$ given by the equation
$$
\frac{4 \rho z}{(1+\rho)^2 (1-z)^2}= \frac{w}{(1-w)^2}  \,.
$$
This function maps the unit disk $U_z$ conformally and univalently onto
the unit disk $U_w$ with the segment $[-1,-\rho]$ removed.
\end{theorem}

\begin{proof} The function $g(z)= f(f(z); \beta)$ is in the class $\mathcal{ B}$ for
a fixed $\beta, 0< \beta <1\,.$ Therefore it satisfies the inequality
\eqref{4.20} from the proof of the previous theorem with $f$ replaced with
$g\,.$ For values of $\beta$ close to one, the image of the disk $U_z$ under
$g$ does not contain any of the circles $\{ w:  |w|=r \}, \rho(\beta)<r<1\,,$
where $\lim_{\beta\to 1} \rho(  \beta)= \rho\,.$ Choosing $\varepsilon$ from the
previous proof to be smaller than $1-\beta\,,$ we conclude that the set
 ${\rm Cr}^{-}_{o} \overline{E}$ contains the segment $[-1,-\rho(\beta)]\,.$
From the monotonicity of the capacity and Theorem \ref{4.1} we obtain
$$
\rc [ ({\rm Cr}^-_{o} \overline{E}) \cap U_w] \ge \rc \, (-1,-\rho(\beta)]
= - \frac{1}{6} {\rm Re}  \frac{S_h(1)}{(h'(1))^2}\,,
$$
where $h(z)=f(z;\rho(\beta))\,.$ Comparison with \eqref{4.20} ($f=g$) gives
  $$
{\rm Re}  \frac{S_g(1)}{(g'(1))^2} \le {\rm Re}  \frac{S_h(1)}{(h'(1))^2}  \,.
$$
Letting $\beta \to 1$ and computing the derivative of the function
$f(z; \rho)$ we arrive at the stated inequality. The proof is complete.
\end{proof}



\begin{theorem}
Let the function $f$ be of class $\mathcal{ B}$ and suppose that the linear Lebesgue measure of the intersection of the image $f(U_z)$ with the imaginary axis is at
most $2t, 0 <t<1\,.$ Then the following inequality holds
$$
{\rm Re}  \frac{S_f(1)}{(f'(1))^2} \le -\frac{3}{2} \left(  \frac{1-t^2}{1+t^2} \right)^2 \,.
$$
The equality holds, for instance, for the function $\tilde{f}_t(z)\,,$ given
by the equation
$$
\frac{2t}{1+t^2} \frac{z}{1-z^2}  =  \frac{w}{1-w^2} \,.
$$
this function maps the unit disk $U_z$ conformally and univalently onto the disk
$U_w$ with the segments $[\pm it, \pm i]\,$ removed.
\end{theorem}

\begin{proof} According to Theorems \ref{4.1} and  \ref{3.8} we have
$$
- \frac{1}{6} \, {\rm Re}  \frac{S_f(1)}{(f'(1))^2} \ge \rc \, E
\ge \rc [  U \setminus {\rm St}(U \setminus E)]
$$
for every set $E$ relatively closed with respect to $U_w$ lying  in the set
$U_w \setminus f(U_z)$ and having a positive distance to the point
$z=1 \,.$ We replace the set $E$ with the intersection of $U_w \setminus f(U_z)$ with the imaginary axis. Then from the monotonicity of the capacity and again
by Theorem \ref{4.1} it follows that
$$
\rc [  U \setminus {\rm St}(U \setminus E)] \ge \rc [  U_w \setminus \tilde{f}_t(U_z)]=
- \frac{1}{6} \, {\rm Re}  \frac{S_{\tilde{f}_t}(1)}{(\tilde{f}_t'(1))^2} \,.
$$
It only remains to compute the derivative on the right hand side
which is an easy exercise. The theorem is proved.
\end{proof}

\bigskip

{\bf Acknowledgements.} The authors are indebted to the referees for their
valuable comments.
This work was completed during the visit of the first author
to the University of Turku, Finland.
\end{document}